\documentclass[11pt]{article}
\usepackage{epsf}
\usepackage{amsbsy,amsmath}
\usepackage{mathtools}
\usepackage{mathrsfs}
\usepackage{amsfonts}
\usepackage{breqn}
\usepackage{enumitem}
\usepackage{enumerate}
\usepackage{amssymb}
\usepackage{enumitem}
\usepackage{eucal}
\usepackage{graphics,mathrsfs}
\usepackage{amsthm}
\usepackage{secdot}
\newtheorem{theorem}{Theorem}[section]
\newtheorem{lemma}[theorem]{Lemma}

\newtheorem{claim}[theorem]{claim}

\theoremstyle{definition}

\theoremstyle{remark}

\numberwithin{equation}{section}

\numberwithin{equation}{section}

\newsavebox{\savepar}

\begin{document}

\title{Existence and multiplicity of solutions to a nonlocal elliptic PDE with variable exponent in a Nehari manifold using the Banach fixed point theorem}
\author{Amita Soni and D. Choudhuri\\ 
\small{Department of Mathematics, National Institute of Technology Rourkela}\\
\small{Emails: soniamita72@gmail.com, dc.iit12@gmail.com}}
\date{}
\maketitle
\begin{abstract}
\noindent In this paper we study the existence and multiplicity of two distinct nontrivial weak solutions of the following equation in Nehari manifold. We have also proved that these solutions are in $L^{\infty}(\Omega)$.
\begin{align*}
\begin{split}
-\Delta_{p(x,y)}^{s(x,y)}u &= \beta|u|^{\alpha(x)-2}u+\lambda f(x,u)\,\,\mbox{in}\,\,\Omega,\\
u &= 0\,\, \mbox{in}\,\, \mathbb{R}^{N}\setminus\Omega
\end{split}
\end{align*}
Here, $\lambda, \beta > 0$ are parameters and $f(x,u)$ is a general nonlinear term satisfying certain conditions. The domain $\Omega\subset\mathbb{R}^N (N\geq 2)$ is smooth and bounded. The relation between the exponents are assumed in the order $2 < \alpha^{-}\leq\alpha(x)\leq\alpha^{+} < p^{-}\leq p(x,y)\leq p^{+} < q^{+} < r^{+} < r^{+2} < p_{s}^{*}(x)$. Also, $\alpha(x)\leq p(x,x)\;\forall\;x\in\overline{\Omega}$ and $s(x,y)p(x,y) < N \;\forall\;(x,y)\in\overline{\Omega}\times\overline{\Omega}$. 
\begin{flushleft}

\end{flushleft}
\end{abstract}
\section{Introduction}
We will consider the following problem.
\begin{align}
\begin{split}
-\Delta_{p(x,y)}^{s(x,y)}u &= \beta|u|^{\alpha(x)-2}u+\lambda f(x,u)\,\,\mbox{in}\,\,\Omega\\
u &= 0\,\, \mbox{in}\,\, \mathbb{R}^{N}\setminus\Omega
\end{split}
\end{align}
where $p\in C(\overline{\Omega}\times\overline{\Omega})$ such that $1 < p(x,y) < \infty$ and $s\in C(\overline{\Omega}\times\overline{\Omega})$ such that $0 < s(x,y) < 1$. Also $s(x,y)p(x,y) < N\;\forall\;(x,y)\in \overline{\Omega}\times\overline{\Omega}$, $\lambda,\beta > 0$ are parameters and the relation between the exponents are assumed to be in the order $2 < \alpha^{-}\leq\alpha(x)\leq\alpha^{+} < p^{-}\leq p(x,y)\leq p^{+} < q^{+} < r^{+} < r^{+2} < p_{s}^{*}(x)$. Also, $\alpha(x)\leq p(x,x)\;\forall\;x\in\overline{\Omega}\;\forall\;x\in\overline{\Omega}$. We have considered an elliptic equation involving a nonlocal type with variable exponent similar to that in \cite{Biswas} but with different assumptions on the nonlinear term. In the literature, there are quite a good number of articles available to show existence of multiple solutions in Nehari manifold for both local and nonlocal operators. For example readers may refer \cite{Ghanmi}, \cite{Choudhuri}, \cite{Brown}.  In \cite{Ghanmi}, the authors dealt with a singular problem involving the fractional Laplace operator in Nehari manifold. In \cite{Choudhuri}, the authors worked with fractional $p$-Laplacian operator involving convex-concave nonlinearities. They have shown the existence of two distinct non trivial weak solutions in Nehari manifold using fiber maps. In \cite{Brown}, the authors have proved the existence of solutions involving Laplacian operator in Nehari manifold. In addition, they have also proved the dependence of Nehari manifold on the parameter $\lambda$ and linked the properties of the manifold to existence and non-existence results of positive solutions. Since there has been an increased growth in PDE's involving the nonlocal operator, advancement of the operators and hence their corresponding space on which we seek solutions for such kind of operators is also increasing rapidly. Some very commonly used areas of research where the nonlocal opeartors are extensively used are in the thin obstacle problems, optimization, anomalous diffusion, finance, phase transition, continuum mechanics, graph theory, jump processes, machine learning etc. Some applications of nonlocal operators can be seen in the articles \cite{Negoro},  \cite{Caf}, \cite{Ruz}, \cite{Laskin}, \cite{Chen}. Nonlocal operators with a variable exponent also finds its use in image restoration and obstacle problems.\\
One of the most interesting and trending nonlocal and nonlinear operator are operators with variable exponent. For more details on these kind of operators and their variants, the readers may refer \cite{Bahrouni}, \cite{Azroul},  \cite{Diening},  \cite{Fan} and the references therein. In \cite{Azroul}, the authors have extended the $W^{s,p(x,y)}$ to a more general fractional case $W^{K,p(x,y)}$ and proved some qualitative properties of this new space. These work includes the basic ideas and origination of variable exponent spaces along with its applications. In past few years, some appreciable work have been done on variable exponent spaces involving Nehari manifold. To see the idea involved in these works, some suggested articles are \cite{Alves}, \cite{Avci}, \cite{Rasouli}, \cite{Fallah} and the references therein.
However due to the presence of variable exponent it is difficult to show the existence of solutions in Nehari manifold in the same manner as done for constant exponent case. This difficulty further increases when the term involved in the equation is a general nonlinear function. To our knowledge there are only few articles in the literature available for the variable exponent problem involving a general nonlinear term whose solutions are in  Nehari manifold. In the problem we have considered in this article, we will in the Nehari manifold. The novelty in this work lies is in the fact that for showing existence of solutions we are applying both the variational technique and the Banach fixed point theorem. We will show the existence of two weak nontrivial solutions and the uniform estimate of both the weak solutions will be discussed. 
\begin{theorem}
Let the nonlinear function $f(x,u)$ satisfies $(f_{1})-(f_{4})$. Then the problem (1.1) admits two distinct nontrivial weak solutions in the Nehari manifold for a particular range of $\lambda$ and $\beta$. Moreover the solutions are in $L^{\infty}(\Omega).$
\end{theorem}
\section{Preliminaries}
We state the well known Banach fixed point theorem which will be used in our work.
\begin{theorem}
\textbf{Banach fixed point theorem}: Let $(X,d)$ be a nonempty complete metric space with a contraction mapping $G:X\rightarrow X$. Then $G$ admits a unique fixed point.
\end{theorem}
\noindent For each open subset $\Omega \subset \mathbb{R}^{N}(N \geq 2),$ let $p\in C(\overline{\Omega}\times\overline{\Omega})$ such that $1 < p(x,y) < \infty$ and $s\in C(\overline{\Omega}\times\overline{\Omega})$ such that $0 < s(x,y) < 1$. Also $s(x,y)p(x,y) < N\;\forall\;(x,y)\in \overline{\Omega}\times\overline{\Omega}$. Let 
\begin{equation}
1 < p^{-}:=\underset{(x,y)\in\overline{\Omega}\times\overline{\Omega}}{\text{inf}}\;p(x,y)\leq \underset{(x,y)\in\overline{\Omega}\times\overline{\Omega}}{\text{sup}}\;p(x,y)=: p^{+} < \infty
\end{equation}
\begin{equation}
0 < s^{-}:=\underset{(x,y)\in\overline{\Omega}\times\overline{\Omega}}{\text{inf}}\;s(x,y)\leq \underset{(x,y)\in\overline{\Omega}\times\overline{\Omega}}{\text{sup}}\;s(x,y)=:s^{+} < 1
\end{equation}
Here, $s(x,y)$ and $p(x,y)$ are symmetric in nature for all $(x,y)\in\overline{\Omega}\times\overline{\Omega}$. The variable exponent space, denoted by $W^{q(x),s(x,y),p(x,y)}(\Omega)$, with variable order and variable exponents is the space 
$$\left\lbrace u\in L^{q(x)}(\Omega):\int_{\Omega}\int_{\Omega}\frac{|u(x)-u(y)|^{p(x,y)}}{\lambda^{p(x,y)}|x-y|^{N+s(x,y)p(x,y)}}dx dy < \infty,\; \text{for some}\; \lambda > 0\right\rbrace$$
with the norm $\|u\|_{s(x,y),p(x,y)}=\|u\|_{q(x)(\Omega)}+[u]^{s(x,y),p(x,y)}_{\Omega},$ where $$[u]^{s(x,y),p(x,y)}_{\Omega}=\text{inf}\left\lbrace \lambda > 0:\int_{\Omega}\int_{\Omega}\frac{|u(x)-u(y)|^{p(x,y)}}{\lambda^{p(x,y)}|x-y|^{N+s(x,y)p(x,y)}}dx dy < 1\right\rbrace.$$
For more details and results related to this space refer \cite{Biswas}. We will now state a few results from the reference \cite{Biswas} which will be used in our article.
\noindent Define $X=W^{q(x),s(x,y),p(x,y)}(\Omega)$ over $T=\mathbb{R}^{2N}\setminus(\Omega^{c}\times\Omega^{c})$ as the space $$\left\lbrace u:\mathbb{R}^N\rightarrow\mathbb{R}:u|_{\Omega}\in L^{q(x)}(\Omega), \int_{T}\frac{|u(x)-u(y)|^{p(x,y)}}{\lambda^{p(x,y)}|x-y|^{N+s(x,y)p(x,y)}}dx dy < \infty,\; \text{for some}\; \lambda > 0\right\rbrace$$
and $X_{0}$ i.e. $W_{0}^{q(x),s(x,y),p(x,y)}(\Omega)$ to be the space $\left\lbrace u\in X:u=0 \;\;\text{a.e. in}\;\; \mathbb{R}^N\setminus\Omega\right\rbrace.$ This space is a convex, reflexive and separable Banach space (for proof refer \cite{Biswas}) with respect to the norm $$\|u\|_{X_{0}}=\text{inf}\left\lbrace \lambda > 0:\int_{T}\frac{|u(x)-u(y)|^{p(x,y)}}{\lambda^{p(x,y)}|x-y|^{N+s(x,y)p(x,y)}}dx dy < 1\right\rbrace$$
$$\;\;\;\;\;\;\;\;\;\;\;\;\;\;\;\;\;\;\;\;=\text{inf}\left\lbrace \lambda > 0:\int_{\mathbb{R}^N}\int_{\mathbb{R}^N}\frac{|u(x)-u(y)|^{p(x,y)}}{\lambda^{p(x,y)}|x-y|^{N+s(x,y)p(x,y)}}dx dy < 1\right\rbrace.$$
\begin{theorem}
Let $\Omega\subset\mathbb{R}^{N}(N\geq 2)$ is a smooth bounded domain. Let $s(.,.),p(.,.)$ satisfy (2.1) and (2.2) along with symmetry such that $s(x,y)p(x,y) < N\;\;\forall\;x,y\in \overline{\Omega}\times\overline{\Omega}$ and $q\in C_{+}(\overline{\Omega})$ such that $p(x,x)\leq q(x) < p^{\ast}_{s}(x)$ for all $x\in\overline{\Omega}$. Suppose that $\beta\in C_{+}(\overline{\Omega})$ such that $1 < \beta(x) < p^{\ast}_{s}(x)$ for $x\in\overline{\Omega}$. Then there exists a constant $K=K(N,s,p,q,\beta,\Omega) > 0$ such that for every $u\in X_{0}, \|u\|_{L^{\beta(x)}(\mathbb{R}^N)}=\|u\|_{L^{\beta(x)}(\Omega)}\leq K\|u\|_{X_{0}}$. Moreover, this embedding is compact.
\end{theorem}
\begin{lemma}
The modular function, $\rho_{X_{0}}(u)=\int_{\mathbb{R}^N}\int_{\mathbb{R}^N}\frac{|u(x)-u(y)|^{p(x,y)}}{|x-y|^{N+s(x,y)p(x,y)}}dx dy$, has the following properties.
\begin{itemize}
\item $\|u\|_{X_{0}} < 1(=1;> 1)\Leftrightarrow \rho_{X_{0}}(u) < 1(=1; >1)$,
\item $\|u\|_{X_{0}} < 1\Rightarrow \|u\|_{X_{0}}^{p^{+}}\leq \rho_{X_{0}}(u)\leq \|u\|_{X_{0}}^{p^{-}},$
\item $\|u\|_{X_{0}} > 1\Rightarrow \|u\|_{X_{0}}^{p^{-}}\leq \rho_{X_{0}}(u)\leq \|u\|_{X_{0}}^{p^{+}},$
\item $\underset{n\rightarrow\infty}{\text{lim}}||u_{n}||_{X_{0}}=0(\infty)\Leftrightarrow \underset{n\rightarrow\infty}{\text{lim}}\rho_{X_{0}}(u_{n})=0(\infty)$.
\end{itemize}
\end{lemma}
\begin{lemma}
Let $\mu(x)\in L^{\infty}(\Omega)$ such that $\mu\geq 0, \mu \not\equiv 0$. Let $\nu : \Omega\rightarrow\mathbb{R}$ be a measurable function such that $\mu(x)\nu(x)\geq 1$ a.e. in $\Omega$. Then for every $u\in L^{\mu(x)\nu(x)}(\Omega)$,
$$\left\||u|^{\mu(.)}\right\|_{L^{\nu(x)}(\Omega)}\leq \left\|u\right\|^{\mu^{-}}_{L^{\mu(x)\nu(x)}(\Omega)}+\left\|u\right\|^{\mu^{+}}_{L^{\mu(x)\nu(x)}(\Omega)}$$
\end{lemma}
\noindent Henceforth, we will denote $\|u\|_{X_{0}}=\|u\|$.\\
We are now stating the two crucial inequalities that will be used while proving the results.
\begin{equation}
|a+b|^{p}\leq 2^{p}(|a|^{p}+|b|^{p})\;\forall\;a,b\in\mathbb{R}\;\text{and}\;1\leq p < \infty
\end{equation}
\begin{equation}
||a|^{p(x)-2}a-|b|^{p(x)-2}b|\leq (p(x)-1)|a-b|(|a|^{p(x)-1}+|b|^{p(x)-1}), p(x) \geq 2
\end{equation}

\section{Functional Analytic Setup}
The problem (1.1) considered as
\begin{align*}
\begin{split}
-\Delta_{p(x,y)}^{s(x,y)}u &= \beta|u|^{\alpha(x)-2}u+\lambda f(x,u)\,\,\mbox{in}\,\,\Omega,\\
u &= 0\,\, \mbox{in}\,\, \mathbb{R}^{N}\setminus\Omega
\end{split}
\end{align*}
has the following assumptions on the nonlinear function $f(x,u)$ .\\
\vspace{1.5mm}
\hspace{-1.7mm}$(f_{1})$ $f(x,0)=0$ and $f(x,u)\in C^{1}(\Omega\times\mathbb{R},\mathbb{R})$ is positively homogeneous of degree $r^{+}-1$ i.e. $f(x,tu)=t^{r^{+}-1}f(x,u)\;\forall\;(x,u)\in\Omega\times\mathbb{R}, t > 0$.\\
\vspace{1.5mm}
\hspace{-1.7mm}$(f_{2})$ $F(x,u):\overline{\Omega}\times\mathbb{R}\rightarrow\mathbb{R}$ is homogeneous of degree $r^{+}$, where $F$ is the primitive of $f$.\\
\vspace{1.5mm}
\hspace{-1.7mm}$(f_{3})$ $\left|\frac{\partial f}{\partial s}\right|\leq\delta|s|^{r^{+}-2}$ for all $x\in\Omega, s\in \mathbb{R}$. Here $\delta$ is a sufficiently small positive number.\\
\vspace{1.5mm}
\hspace{-1.7mm}$(f_{4})$ $\underset{s\rightarrow\infty}{\text{lim}}\frac{f(x,s)}{s^{r^{+}-1}}=0$ uniformly a.e. $x\in\Omega$.\\
\noindent Here, $s(x,y)p(x,y) < N \;\forall\;(x,y)\in\overline{\Omega}\times\overline{\Omega}$ and $2 < \alpha^{-}\leq\alpha(x)\leq\alpha^{+} < p^{-}\leq p(x,y)\leq p^{+} < q^{+} < r^{+} < r^{+2} < p_{s}^{*}(x)$. Also, $\alpha(x)\leq p(x,x)\;\forall\;x\in\overline{\Omega}$. The functional corresponding to the problem (1.1) is defined as
\begin{align*}
I_{\lambda,\beta}(u)=\int_{\mathbb{R}^N}\int_{\mathbb{R}^N}\frac{|u(x)-u(y)|^{p(x,y)}}{p(x,y)|x-y|^{N+s(x,y)p(x,y)}}dx dy-\beta\int_{\Omega}\frac{|u|^{\alpha(x)}}{\alpha(x)}dx-\lambda\int_{\Omega}F(x,u)dx
\end{align*}
A function $u\in W_{0}^{s(x,y),p(x,y)}(\overline{\Omega})$ is said to be a weak solution to the problem (1.1) if $\forall\;v\in W_{0}^{s(x,y),p(x,y)}(\overline{\Omega})$ it satisfies the following.
\begin{align}
\begin{split}
\int_{\mathbb{R}^{N}}\int_{\mathbb{R}^{N}}\frac{|u(x)-u(y)|^{p(x,y)-2}(u(x)-u(y))(v(x)-v(y))}{|x-y|^{N+s(x,y)p(x,y)}}dx dy=\beta\int_{\Omega}|u|^{\alpha(x)}uv dx+\lambda f(x,u)v dx
\end{split}
\end{align}
The fiber maps corresponding to the functional $I_{\lambda,\beta}$ and its derivatives are defined as follows.
\begin{align*}
\begin{split}
I_{\lambda,\beta}(tu)\;\;\;&=\int_{\mathbb{R}^N}\int_{\mathbb{R}^N}\frac{t^{p(x,y)}{|u(x)-u(y)|}^{p(x,y)}}{p(x,y)|x-y|^{N+s(x,y)p(x,y)}}dx dy-\beta\int_{\Omega}\frac{t^{\alpha(x)}|u|^{\alpha(x)}}{\alpha(x)}dx\\
\MoveEqLeft\hspace{9mm}-\lambda t^{r^+}\int_{\Omega}F(x,u)dx\\
\frac{d}{dt}I_{\lambda,\beta}(tu)&=\int_{\mathbb{R}^N}\int_{\mathbb{R}^N}\frac{t^{p(x,y)-1}{|u(x)-u(y)|}^{p(x,y)}}{|x-y|^{N+s(x,y)p(x,y)}}dx dy-\beta\int_{\Omega}t^{\alpha(x)-1}|u|^{\alpha(x)}dx\\
\MoveEqLeft\hspace{9mm}-\lambda r^{+} t^{r^{+}-1}\int_{\Omega}F(x,u)dx\\
\left.\frac{d}{dt}I_{\lambda,\beta}(tu)\right|_{t=1}&=\int_{\mathbb{R}^N}\int_{\mathbb{R}^N}\frac{{|u(x)-u(y)|}^{p(x,y)}}{|x-y|^{N+s(x,y)p(x,y)}}dx dy-\beta\int_{\Omega}|u|^{\alpha(x)}dx-\lambda r^{+}\int_{\Omega}F(x,u)dx\\
\end{split}
\end{align*}
\begin{align*}
\begin{split}
\left.\frac{d^2}{dt^2}I_{\lambda,\beta}(tu)\right|_{t=1}&=\int_{\mathbb{R}^N}\int_{\mathbb{R}^N}(p(x,y)-1)\frac{{|u(x)-u(y)|}^{p(x,y)}}{|x-y|^{N+s(x,y)p(x,y)}}dx dy-\beta\int_{\Omega}(\alpha(x)-1)|u|^{\alpha(x)}dx\\
&-\lambda r^{+}(r^{+}-1)\int_{\Omega}F(x,u)dx
\end{split}
\end{align*}
\section{Existence Results}
We will first define the Nehari manifold as
$\mathcal{N}=\left\lbrace u\in X_{0}\setminus\left\lbrace 0\right\rbrace :\langle I^{\prime}_{\lambda,\beta}(u),u\rangle=0\right\rbrace.$ We will also define 
$$\mathcal{N}^{+}= \left\lbrace u\in\mathcal{N}:\frac{d^2}{dt^2}I_{\lambda,\beta}(tu)|_{t=1} > 0\right\rbrace ,$$
$$\mathcal{N}^{0}= \left\lbrace u\in\mathcal{N}:\frac{d^2}{dt^2}I_{\lambda,\beta}(tu)|_{t=1} = 0\right\rbrace $$ and $$\mathcal{N}^{-}= \left\lbrace u\in\mathcal{N}:\frac{d^2}{dt^2}I_{\lambda,\beta}(tu)|_{t=1} < 0\right\rbrace .$$ It can be apparently seen that if $u \in X_{0}, t_{0} > 0$ then $t_{0}u\in\mathcal{N}$ iff $\frac{d}{dt}I_{\lambda,\beta}(t_{0}u)=0$. We will prove all the upcoming results for $\left\|u\right\| < 1$.
\begin{lemma}\label{lem1}
The functional $I_{\lambda,\beta}$ is coercive and bounded below over $\mathcal{N}$.
\end{lemma}
\begin{proof}
Suppose $u\in \mathcal{N}$. Then we have the following
\begin{align}
\int_{\mathbb{R}^N}\int_{\mathbb{R}^N}\frac{|u(x)-u(y)|^{p(x,y)}}{|x-y|^{N+s(x,y)p(x,y)}}dx dy-\beta\int_{\Omega}|u|^{\alpha(x)}dx-\lambda r^{+}\int_{\Omega}F(x,u)dx=0.
\end{align}
Using equation (4.1) in the functional $I_{\lambda,\beta}$, we get
\begin{align*}
\begin{split}
I_{\lambda,\beta}(u)=&\int_{\mathbb{R}^N}\int_{\mathbb{R}^N}\frac{|u(x)-u(y)|^{p(x,y)}}{p(x,y)|x-y|^{N+s(x,y)p(x,y)}}dx dy-\beta\int_{\Omega}\frac{|u|^{\alpha(x)}}{\alpha(x)}dx-\lambda\int_{\Omega}F(x,u)dx\\
\geq& \frac{1}{p^{+}}\int_{\mathbb{R}^N}\int_{\mathbb{R}^N}\frac{|u(x)-u(y)|^{p(x,y)}}{|x-y|^{N+s(x,y)p(x,y)}}dx dy-\frac{\beta}{\alpha^{-}}\int_{\Omega}|u|^{\alpha(x)}dx-\lambda\int_{\Omega}F(x,u)dx\\
\geq&\left(\frac{1}{p^+}-\frac{1}{r^+}\right)\int_{\mathbb{R}^N}\int_{\mathbb{R}^N}\frac{|u(x)-u(y)|^{p(x,y)}}{|x-y|^{N+s(x,y)p(x,y)}}dx dy+\beta\left(\frac{1}{r^+}-\frac{1}{\alpha^-}\right)\int_{\Omega}|u|^{\alpha(x)}dx
\end{split}
\end{align*}
Choose $\beta_{1}$ small enough such that for every $\beta\in(0,\beta_{1})$ we get $I_{\lambda,\beta}(u) \geq c\left\|u\right\|^{p^+}$, where $c > 0$. Hence, $I_{\lambda,\beta}$ is coercive and bounded below over $\mathcal{N}$.
\end{proof}
\begin{lemma}
The set $\mathcal{N}^{0}=\left\lbrace u\in\mathcal{N}:\frac{d^2}{dt^{2}}I_{\lambda,\beta}(tu)|_{t=1}=0\right\rbrace =\phi$.
\end{lemma}
\begin{proof}
The proof is by contradiction. Let there exist a $u_{0}(\neq 0)\in \mathcal{N}$ such that\\ $\frac{d^2}{dt^{2}}I_{\lambda,\beta}(tu_{0})|_{t=1}=0$. From this we get that 
\begin{align}
\begin{split}
&\int_{\mathbb{R}^N}\int_{\mathbb{R}^N}(p(x,y)-1)\frac{|u_{0}(x)-u_{0}(y)|^{p(x,y)}}{|x-y|^{N+s(x,y)p(x,y)}}dx dy-\beta\int_{\Omega}(\alpha(x)-1)|u_{0}|^{\alpha(x)}dx\\
&-\lambda r^{+}(r^{+}-1)\int_{\Omega}F(x,u_{0})dx=0.
\end{split}
\end{align}
Since $u_{0}\in\mathcal{N}$ so we have
\begin{align}
\int_{\mathbb{R}^N}\int_{\mathbb{R}^N}\frac{|u_{0}(x)-u_{0}(y)|^{p(x,y)}}{|x-y|^{N+s(x,y)p(x,y)}}dx dy-\beta\int_{\Omega}|u_{0}|^{\alpha(x)}dx-\lambda r^{+}\int_{\Omega}F(x,u_{0})dx=0
\end{align}
Using (4.3) in (4.2), we get
\begin{align*}
(p^{+}-r^{+})\int_{\mathbb{R}^N}\int_{\mathbb{R}^N}\frac{|u_{0}(x)-u_{0}(y)|^{p(x,y)}}{|x-y|^{N+s(x,y)p(x,y)}}dx dy+\beta(r^{+}-\alpha^{-})\int_{\Omega}|u_{0}|^{\alpha(x)}dx \geq 0
\end{align*}
Thus, 
\begin{align*}
\left\|u_{0}\right\|^{p^{+}}\leq \frac{\beta(r^{+}-\alpha^{-})\int_{\Omega}|u_{0}|^{\alpha(x)}dx}{(r^{+}-p^{+})}\;\forall\;\beta > 0.
\end{align*}
This implies $\left\|u_{0}\right\|=0$ which is a contradiction to our assumption that $\|u_{0}\|\neq 0$. Hence, the set $\mathcal{N}^{0}$ is empty.
\end{proof}
\noindent Define a function $$\phi(t)= \int_{\mathbb{R}^N}\int_{\mathbb{R}^N}t^{p(x,y)-r^{+}}\frac{|u(x)-u(y)|^{p(x,y)}}{|x-y|^{N+s(x,y)p(x,y)}}dx dy-\beta\int_{\Omega}t^{\alpha(x)-r^{+}}|u|^{\alpha(x)}dx.$$
On differentiating $\phi$ w.r.t $t$ we get 
\begin{align*}
\begin{split}
\phi^{\prime}(t)&= \int_{\mathbb{R}^N}\int_{\mathbb{R}^N}(p(x,y)-r^{+})t^{p(x,y)-r^{+}-1}\frac{|u(x)-u(y)|^{p(x,y)}}{|x-y|^{N+s(x,y)p(x,y)}}dx dy\\
&-\beta\int_{\Omega}(\alpha(x)-r^{+})t^{\alpha(x)-r^{+}}|u|^{\alpha(x)}dx.
\end{split}
\end{align*}
It is easy to see that $tu\in\mathcal{N}$ iff $\phi(t)=\lambda r^{+}\int_{\Omega}F(x,u)dx$. When $tu\in\mathcal{N}$ and $\frac{d}{dt}I_{\lambda,\beta}(tu)=0$, we have
\begin{align}
\lambda t^{r^{+}-1}r^{+}\int_{\Omega}F(x,u)dx=\int_{\mathbb{R}^N}\int_{\mathbb{R}^N}t^{p(x,y)-1}\frac{|u(x)-u(y)|^{p(x,y)}}{|x-y|^{N+s(x,y)p(x,y)}}dx dy-\beta\int_{\Omega}t^{\alpha(x)-1}|u|^{\alpha(x)}dx.
\end{align}
Using (4.4), we get
\begin{align*}
\begin{split}
\frac{d^2}{dt^2}I_{\lambda,\beta}(tu)=&\int_{\mathbb{R}^N}\int_{\mathbb{R}^N}(p(x,y)-1)t^{p(x,y)-2}\frac{|u(x)-u(y)|^{p(x,y)}}{|x-y|^{N+s(x,y)p(x,y)}}dx dy-(r^{+}-1)t^{-1}\\&\left(\int_{\mathbb{R}^N}\int_{\mathbb{R}^N}t^{p(x,y)-1}\frac{|u(x)-u(y)|^{p(x,y)}}{|x-y|^{N+s(x,y)p(x,y)}}dx dy-\beta\int_{\Omega}t^{\alpha(x)-1}|u|^{\alpha(x)}dx\right)\\
&-\beta\int_{\Omega}(\alpha(x)-1)t^{\alpha(x)-2}|u|^{\alpha(x)}dx\\
=&\int_{\mathbb{R}^N}\int_{\mathbb{R}^N}(p(x,y)-r^{+})t^{p(x,y)-2}\frac{|u(x)-u(y)|^{p(x,y)}}{|x-y|^{N+s(x,y)p(x,y)}}dx dy\\
&-\beta\int_{\Omega}(\alpha(x)-r^{+})t^{\alpha(x)-2}|u|^{\alpha(x)}dx
\end{split}
\end{align*} 
Hence, $t^{r^{+}-1}\phi^{\prime}(t)=\frac{d^2}{dt^2}I_{\lambda,\beta}(tu)$. Let us assume $t\in(0,1)$. Thus,
\begin{align*}
\begin{split}
\phi^{\prime}(t)&=\int_{\mathbb{R}^N}\int_{\mathbb{R}^N}(p(x,y)-r^{+})t^{p(x,y)-r^{+}-1}\frac{|u(x)-u(y)|^{p(x,y)}}{|x-y|^{N+s(x,y)p(x,y)}}dx dy\\
&-\beta\int_{\Omega}(\alpha(x)-r^{+})t^{\alpha(x)-r^{+}-1}|u|^{\alpha(x)}dx
\end{split}
\end{align*}
A simple computation leads to the fact that when $t > \left[\frac{\beta(r^{+}-\alpha^{+})\int_{\Omega}|u|^{\alpha(x)}dx}{(r^{+}-p^{-})\left\|u\right\|^{p^-}}\right]^{\frac{1}{p^{+}-\alpha^{-}}}=t_{1}$(say) then $\phi^{\prime}(t) < 0$ and $\phi^{\prime}(t) \geq 0$ when $t \leq \left[\frac{\beta(r^{+}-\alpha^{+})\int_{\Omega}|u|^{\alpha(x)}dx}{(r^{+}-p^{-})\left\|u\right\|^{p^+}}\right]^{\frac{1}{p^{-}-\alpha^{+}}}=t_{2}$(say).\\
Now there are two possibilities either $t_{1} > t_{2}$ or $t_{2} > t_{1}$. Since, $\phi^{\prime}(t)$ is continuous so in both the cases $\exists$ a $\tilde{t}$ such that $\phi^{\prime}(\tilde{t})=0$ and $\tilde{t}$ is a maximum point of $\phi(t)$. Further choose $\lambda > 0$ say $\lambda_{1}$ such that $\phi(\tilde{t}) > \lambda r^{+}\int_{\Omega}F(x,u)$. So, there exists $\tau_{1}, \tau_{2}$ in  the neighbourhood of $\tilde{t}$ such that $\phi^{\prime}(\tau_{1}u) > 0$ and $\phi^{\prime}(\tau_{2}u) < 0$. This implies $\frac{d^2}{dt^2}I_{\lambda,\beta}(\tau_{1}u) > 0$ and $\frac{d^2}{dt^2}I_{\lambda,\beta}(\tau_{2}u) < 0$ i.e. $\tau_{1}u\in \mathcal{N}^{+}$ and $\tau_{2}u\in \mathcal{N}^{-}$.\\

\noindent\textbf{Remark}: Since, the set $\mathcal{N}^{0}$ is empty so $\mathcal{N}=\mathcal{N}^{+}\cup\mathcal{N}^{-}$. Also $I_{\lambda,\beta}$ is bounded below on $\mathcal{N}$ so it is bounded below on both $\mathcal{N}^{+}$ and $\mathcal{N}^{-}$.
\begin{lemma}
There exists a minimizer of $I_{\lambda,\beta}$ in $\mathcal{N}^{+}$ which is also a solution to the problem (1.1) for $\beta\in(0,\beta_{1})$.
\end{lemma}
\begin{proof}
Let $i^{+}=\underset{u\in\mathcal{N}^{+}}{\textrm{inf}}\left\lbrace I_{\lambda,\beta}(u)\right\rbrace$. Since, $I_{\lambda,\beta}$ is bounded below on $\mathcal{N}^{+}$, there exists a minimizing sequence say $(u_{n})$ in $\mathcal{N}^{+}$ such that $I_{\lambda,\beta}(u_{n})\rightarrow i^{+}$ as $n \rightarrow\infty$. Due to coercivity of $I_{\lambda,\beta}$ on $\mathcal{N}$, $(u_{n})$ is bounded in $\mathcal{N}\subset Y$. Since $Y$ is reflexive so $u_{n}\rightharpoonup u_{1}$(say) in $Y$ and by compact embedding, $u_{n}\rightarrow u_{1}$ in $L^{\alpha(x)}(\Omega)$. We will now show that $u_{n}\rightarrow u_{1}$ in $Y$. For if not, then by the weak lower semicontinuity of the norm we have
\begin{equation}
\int_{\mathbb{R}^N}\int_{\mathbb{R}^N}\frac{|u_{1}(x)-u_{1}(y)|^{p(x,y)}}{|x-y|^{N+s(x,y)p(x,y)}}dx dy \leq \underset{n\rightarrow\infty}{\underline{\textrm{lim}}}\int_{\mathbb{R}^N}\int_{\mathbb{R}^N}\frac{|u_{n}(x)-u_{n}(y)|^{p(x,y)}}{|x-y|^{N+s(x,y)p(x,y)}}dx dy
\end{equation}
Using the fact that $(u_{n})\in\mathcal{N}$, we get,
\begin{align*}
\begin{split}
I_{\lambda,\beta}(u_{n})&\geq \frac{1}{p^{+}}\int_{\mathbb{R}^N}\int_{\mathbb{R}^N}\frac{|u_{n}(x)-u_{n}(y)|^{p(x,y)}}{|x-y|^{N+s(x,y)p(x,y)}}dx dy-\frac{\beta}{\alpha^{-}}\int_{\Omega}|u_{n}|^{\alpha(x)}dx-\lambda\int_{\Omega}F(x,u_{n})dx\\
&=\frac{1}{p^{+}}\int_{\mathbb{R}^N}\int_{\mathbb{R}^N}\frac{|u_{n}(x)-u_{n}(y)|^{p(x,y)}}{|x-y|^{N+s(x,y)p(x,y)}}dx dy-\frac{\beta}{\alpha^{-}}\int_{\Omega}|u_{n}|^{\alpha(x)}dx+\frac{1}{r^{+}}\left[\beta\int_{\Omega}|u_{n}|^{\alpha(x)}dx\right.\\
&\left.\;\;\;\;-\int_{\mathbb{R}^N}\int_{\mathbb{R}^N}\frac{|u_{n}(x)-u_{n}(y)|^{p(x,y)}}{|x-y|^{N+s(x,y)p(x,y)}}dx dy\right] \\\
&=\left(\frac{1}{p^{+}}-\frac{1}{r^{+}}\right)\int_{\mathbb{R}^N}\int_{\mathbb{R}^N}\frac{|u_{n}(x)-u_{n}(y)|^{p(x,y)}}{|x-y|^{N+s(x,y)p(x,y)}}dx dy+\beta\left(\frac{1}{r^{+}}-\frac{1}{\alpha^{-}}\right)\int_{\Omega}|u_{n}|^{\alpha(x)}dx
\end{split}
\end{align*}
Letting limit $n\rightarrow\infty$ both sides, $\beta\in(0,\beta_{1})$ and using (4.5),
\begin{align*}
\begin{split}
i^{+}&\geq\left(\frac{1}{p^{+}}-\frac{1}{r^{+}}\right)\underset{n\rightarrow\infty}{\text{lim}}\int_{\mathbb{R}^N}\int_{\mathbb{R}^N}\frac{|u_{n}(x)-u_{n}(y)|^{p(x,y)}}{|x-y|^{N+s(x,y)p(x,y)}}dx dy+\beta\left(\frac{1}{r^{+}}-\frac{1}{\alpha^{-}}\right)\int_{\Omega}|u_{1}|^{\alpha(x)}dx\\
&\geq \left(\frac{1}{p^{+}}-\frac{1}{r^{+}}\right)\underset{n\rightarrow\infty}{\text{\underline{lim}}}\int_{\mathbb{R}^N}\int_{\mathbb{R}^N}\frac{|u_{n}(x)-u_{n}(y)|^{p(x,y)}}{|x-y|^{N+s(x,y)p(x,y)}}dx dy+\beta\left(\frac{1}{r^{+}}-\frac{1}{\alpha^{-}}\right)\int_{\Omega}|u_{1}|^{\alpha(x)}dx\\
& \geq \left(\frac{1}{p^{+}}-\frac{1}{r^{+}}\right)\int_{\mathbb{R}^N}\int_{\mathbb{R}^N}\frac{|u_{1}(x)-u_{1}(y)|^{p(x,y)}}{|x-y|^{N+s(x,y)p(x,y)}}dx dy+\beta\left(\frac{1}{r^{+}}-\frac{1}{\alpha^{-}}\right)\int_{\Omega}|u_{1}|^{\alpha(x)}dx\\
& > 0
\end{split}
\end{align*}
We already know that there exists $\tau_{1}u\in\mathcal{N}^{+}$ such that $I_{\lambda,\beta}(\tau_{1}u) < 0$. This is a contradiction to our assumption that $u_{n}$ is not convergent to $u_{1}$. Thus, $u_{n}\rightarrow u_{1}$ in $Y$ and $I_{\lambda,\beta}(u_{1})=\underset{n\rightarrow\infty}{\text{lim}}I_{\lambda,\beta}(u_{n})=\underset{u\in\mathcal{N}^{+}}{\textrm{inf}}\left\lbrace I_{\lambda,\beta}(u)\right\rbrace$. This proves that $I_{\lambda,\beta}$ has a minimizer in $\mathcal{N}^{+}$.
\end{proof}
\begin{lemma}
There exists a minimizer of $I_{\lambda,\beta}$ in $\mathcal{N}^{-}$ which is also a solution of the problem for $\beta\in(0,\beta_{2})$.
\end{lemma}
\noindent We now state and prove some lemmas.
\begin{lemma}
Let $p(x,y),s(x,y)$ as in Theorem 2.3. Then for each $f$ in $W_{0}^{-s(x,y),p^{\prime}(x,y)}(\Omega)$, the following problem 
\begin{align}
\begin{split}
(-\Delta)_{p(x,y)}^{s(x,y)}u &= f\,\,\mbox{in}\,\,\Omega,\\
u &= 0\,\, \mbox{in}\,\, \mathbb{R}^{N}\setminus\Omega
\end{split}
\end{align}
has a unique weak solution.
\end{lemma}
\begin{proof}
Define $A_{p(x,y)}(u,v)=\int_{\mathbb{R}^N}\int_{\mathbb{R}^N}\frac{|u(x)-u(y)|^{p(x,y)-2}(u(x)-u(y))(v(x)-v(y))}{|x-y|^{N+s(x,y)p(x,y)}}dx dy$.
Fix\\ $u\in W_{0}^{s(x,y),p(x,y)}(\overline{\Omega})$. Then using H\"{o}lder's inequality, we get for every $v\in W_{0}^{s(x,y),p(x,y)}(\overline{\Omega})$,
\begin{equation*}
|A_{p(x,y)}(u,v)|\leq
\begin{cases}
\left\|u\right\|^{p^{-}-1}\left\|v\right\|,& \text{if}\;\;{\left\|u\right\|<1}\\
\left\|u\right\|^{p^{+}-1}\left\|v\right\|,& \text{if}\;\;{\left\|u\right\|>1}.
\end{cases}
\end{equation*}
Hence, $A_{p(x,y)}(u,v)$ is well defined and bounded. Also,$A_{p(x,y)}(u,.)$ is linear in second variable and so $A_{p(x,y)}(u,.)\in W^{-s(x,y),p^{\prime}(x,y)}(\Omega)$. We next prove the coercivity of $A_{p(x,y)}$. We have that
$$A_{p(x,y)}(u,u)=\int_{\mathbb{R}^N}\int_{\mathbb{R}^N}\frac{|u(x)-u(y)|^{p(x,y)}}{|x-y|^{N+s(x,y)p(x,y)}}dx dy.$$
If $\left\|u\right\|\leq 1$, then $A_{p(x,y)}(u,u)\geq \left\|u\right\|^{p^+}$ and if $\left\|u\right\| > 1$, then $A_{p(x,y)}(u,u)> \left\|u\right\|^{p^-}$. Hence, $A_{p(x,y)}$ is coercive.  We now show that $A_{p(x,y)}$ is strictly monotone. To prove this we first need to prove an inequality which is as follows.
\begin{equation}
(|a|^{p(x,y)-2}a-|b|^{p(x,y)-2}b)(a-b)\geq 
\begin{cases}
C(p)|a-b|^{p^+} &, \text{if}\;\;{|a-b|\leq 1}\\
C(p)|a-b|^{p^-} &, \text{if}\;\;{|a-b|>1}.
\end{cases}
\end{equation}
\begin{proof}
Let $J(p(x,y))=(|a|^{p(x,y)-2}a-|b|^{p(x,y)-2}b)(a-b)$. Define a function $$g(t)=\left(|ta+(1-t)b|\right)^{p(x,y)-2}\\
\left(ta+(1-t)b\right)$$ then $g(0)=|b|^{p(x,y)-2}b$ and $g(1)=|a|^{p(x,y)-2}a$. So, $J(p(x,y))$ can also be represented as
\begin{align*}
\begin{split}
J(p(x,y))&=(g(1)-g(0))(a-b)=(a-b)\int_{0}^{1}g^{\prime}(t)dt\\
&=(a-b)\int_{0}^{1}\left[|ta+(1-t)b|^{p(x,y)-2}(a-b)+(p(x,y)-2)|ta+(1-t)b|^{p(x,y)-4}\right.\\
&\left.\;\;\;\;({ta+(1-t)b})^2(a-b)\right]dt
\end{split}
\end{align*}
Since, $p(x,y) > 2$ so $J(p(x,y))\geq \int_{0}^{1}|ta+(1-t)b|^{p(x,y)-2}(a-b)^{2}dt$ i.e. $J(p(x,y))\geq (a-b)^{2}\int_{0}^{1}|ta+(1-t)b|^{p(x,y)-2}dt.$ We can further write 
\begin{align*}
\begin{split}
ta+(1-t)b&=ta+(1-t)b+(1-t)a-(1-t)a\\
&=(1-t)(b-a)+a\\
&=a-(1-t)(b-a)
\end{split}
\end{align*}
So,
$|ta+(1-t)b|=|a-(1-t)(b-a)|\geq \left\lvert|a|-(1-t)|(a-b)|\right\rvert$.\\
 Now there arises two cases.\\
\textbf{Case I}: If $|a|\geq |b-a|$ then 
\begin{align*}
\begin{split}
|ta+(1-t)b|&\geq \left\lvert|a|-(1-t)|(a-b)|\right\rvert\\
&\geq |a|-(1-t)|(a-b)|\\
&\geq |b-a|-(1-t)|b-a|\\
&=t|b-a|\\
&=t|a-b|
\end{split}
\end{align*}
Hence, as $t\in(0,1)$,
\begin{align*}
\begin{split}
J(p(x,y))&\geq (a-b)^{2}\int_{0}^{1}(t|a-b|)^{p(x,y)-2}dt\\
&=\int_{0}^{1}t^{p(x,y)-2}|a-b|^{p(x,y)}dt \\
&\geq \int_{0}^{1}t^{p^{+}-2}|a-b|^{p(x,y)}dt
\end{split}
\end{align*}
(i) If $|a-b|\leq 1$ then $J(p(x,y))\geq\frac{|a-b|^{p^{+}}}{p^{+}-1}$.\\
(ii) If $|a-b| > 1$ then $J(p(x,y))\geq\frac{|a-b|^{p^{-}}}{p^{+}-1}$.\\

\noindent\textbf{case II}: If $|a| < |b-a|$ then $|ta+(1-t)b|^{2}\leq 2^{2}|a-b|^{2}$.
\begin{align*}
\begin{split}
J(p(x,y))&\geq (a-b)^{2}\int_{0}^{1}|ta+(1-t)b|^{p(x,y)-2}dt\\
&=(a-b)^{2}\int_{0}^{1}\frac{|ta+(1-t)b|^{p(x,y)}}{|ta+(1-t)b|^2}dt\\
&\geq \frac{1}{2^2}\int_{0}^{1}|ta+(1-t)b|^{p(x,y)}dt\\
&=\frac{1}{2^2}\int_{0}^{1}(|ta+(1-t)b|^2)^{\frac{p(x,y)}{2}}dt\\
&\geq\frac{1}{2^2}\left(\frac{|a-b|^2}{3}\right)^{\frac{p(x,y)}{2}}
\end{split}
\end{align*}
(i) If $|a-b|\leq 1$ then $J(p(x,y))\geq\frac{|a-b|^{p^{+}}}{2^{2}\cdot3^{\frac{p^+}{2}}}$.\\
(ii) If $|a-b| > 1$ then $J(p(x,y))\geq\frac{|a-b|^{p^{-}}}{2^{2}\cdot3^{\frac{p^+}{2}}}$.\\
Hence, 
\begin{equation*}
J(p(x,y))\geq 
\begin{cases}
C(p)|a-b|^{p^+} ,&\text{if}\;\; {|a-b|\leq 1}\\
C(p)|a-b|^{p^-} ,&\text{if}\;\; {|a-b|> 1}
\end{cases}
\end{equation*}
where $C(p)=\text{min}\left\lbrace\frac{1}{p^{+}-1},\frac{1}{2^{2}\cdot3^{\frac{p^+}{2}}}\right\rbrace$ and hence the inequality is proved.
\end{proof}
\noindent Let $F_{p(x,y)}(u,v)=|u(x)-u(y)|^{p(x,y)-2}(u(x)-u(y))(v(x)-v(y))$. Define sets\\
 $A_{\alpha}=\left\lbrace x\in\mathbb{R}^{N}:|(u-v)(x)|>\alpha\right\rbrace$ and $B_{\alpha}=\left\lbrace x\in\mathbb{R}^{N}: 0\leq |(u-v)(x)|\leq\alpha\right\rbrace=\mathbb{R}^{N}\setminus A_{\alpha}$. Now, decompose $A_{\alpha}$ into $A_{\alpha}^{+}$and $A_{\alpha}^{-}$ where 
 $$A_{\alpha}^{+}=\left\lbrace x\in\mathbb{R}^{N}:(u-v)(x)>\alpha\right\rbrace \;\text{and}\; A_{\alpha}^{-}=\left\lbrace x\in\mathbb{R}^{N}:(u-v)(x)<-\alpha\right\rbrace$$ such that $A_{\alpha}=A_{\alpha}^{+}\cup A_{\alpha}^{-}$. Further decompose $B_{\alpha}$ into 
 $$B_{\alpha}^{+}=\left\lbrace x\in\mathbb{R}^{N}:0\leq(u-v)(x)\leq\alpha\right\rbrace \;\text{and}\; B_{\alpha}^{-}=\left\lbrace x\in\mathbb{R}^{N}:-\alpha\leq(u-v)(x)\leq 0\right\rbrace$$ such that $B_{\alpha}=B_{\alpha}^{+}\cup B_{\alpha}^{-}$.
Define $w_{\alpha}=(|u-v|-\alpha)^{+}\text{sgn}(u-v); \alpha\geq 0$.\\
\noindent We will first prove that $w_{\alpha}\in W_{0}^{s(x,y),p(x,y)}(\Omega)$ when $u-v \in W_{0}^{s(x,y),p(x,y)}(\Omega).$\\
(a) Let $u\in W_{0}^{s(x,y),p(x,y)}(\overline{\Omega})$ and hence $u\in L^{q(x)}(\overline{\Omega})$. 
$$\int_{\Omega}\int_{\Omega}\frac{\mid\;|u(x)|-|u(y)|\; \mid^{p(x,y)}}{\lambda^{p(x,y)}|x-y|^{N+s(x,y)p(x,y)}}dx dy\leq\int_{\Omega}\int_{\Omega}\frac{|u(x)-u(y)|^{p(x,y)}}{\lambda^{p(x,y)}|x-y|^{N+s(x,y)p(x,y)}}dx dy < \infty.$$
This implies that $|u|\in W_{0}^{s(x,y),p(x,y)}(\overline{\Omega})$.\\
(b) $0 < u^{+}\leq|u|$ hence $u^{+}\in L^{q(x)}(\overline{\Omega})$.\\
Let $A=\left\lbrace x\in(\overline{\Omega}):u(x)\geq 0\right\rbrace$ and $B=\left\lbrace x\in(\overline{\Omega}):u(x) < 0\right\rbrace$.
\begin{align*}
\begin{split}
\int_{\Omega}\int_{\Omega}\frac{|u^{+}(x)-u^{+}(y)|^{p(x,y)}}{\lambda^{p(x,y)}|x-y|^{N+s(x,y)p(x,y)}}dx dy=&\int_{A}\int_{A}\frac{|u(x)-u(y)|^{p(x,y)}}{\lambda^{p(x,y)}|x-y|^{N+s(x,y)p(x,y)}}dx dy\\
&+\int_{B}\int_{A}\frac{|u(y)|^{p(x,y)}}{\lambda^{p(x,y)}|x-y|^{N+s(x,y)p(x,y)}}dx dy\\
&+\int_{A}\int_{B}\frac{|u(x)|^{p(x,y)}}{\lambda^{p(x,y)}|x-y|^{N+s(x,y)p(x,y)}}dx dy
\end{split}
\end{align*}
Since, $x\in A, y\in B$ we have $|u(y)|\leq|u(x)-u(y)|$ and $|u(x)|\leq|u(x)-u(y)|$. Hence
$$\int_{\Omega}\int_{\Omega}\frac{|u^{+}(x)-u^{+}(y)|^{p(x,y)}}{\lambda^{p(x,y)}|x-y|^{N+s(x,y)p(x,y)}}dx dy\leq\int_{\Omega}\int_{\Omega}\frac{|u(x)-u(y)|^{p(x,y)}}{\lambda^{p(x,y)}|x-y|^{N+s(x,y)p(x,y)}}dx dy < \infty.$$
This implies that $u^{+}\in W_{0}^{s(x,y),p(x,y)}(\overline{\Omega})$.\\
(c) Since, $|u|=u^{+}+u^{-}$ so $u^{-}\in W_{0}^{s(x,y),p(x,y)}(\overline{\Omega})$.\\
Hence, $w_{\alpha}\in W_{0}^{s(x,y),p(x,y)}(\overline{\Omega})$. It can be observed that in $B_{\alpha}, w_{\alpha}=0$. Now, we have $A_{p(x,y)}(u,w_{\alpha})-A_{p(x,y)}(v,w_{\alpha})$ 
\begin{align*}
\begin{split}
=&\int_{\mathbb{R}^N}\int_{\mathbb{R}^N}\frac{|u(x)-u(y)|^{p(x,y)-2}(u(x)-u(y))(w_{\alpha}(x)-w_{\alpha}(y))}{|x-y|^{N+s(x,y)p(x,y)}}dx dy\\
-&\int_{\mathbb{R}^N}\int_{\mathbb{R}^N}\frac{|v(x)-v(y)|^{p(x,y)-2}(v(x)-v(y))(w_{\alpha}(x)-w_{\alpha}(y))}{|x-y|^{N+s(x,y)p(x,y)}}dx dy\\
=&\int_{\mathbb{R}^N}\int_{\mathbb{R}^N}\frac{F_{p(x,y)}(u,w_{\alpha})-F_{p(x,y)}(v,w_{\alpha})}{|x-y|^{N+s(x,y)p(x,y)}}dx dy\\
=&\int_{A_{\alpha}}\int_{A_{\alpha}}\frac{F_{p(x,y)}(u,w_{\alpha})-F_{p(x,y)}(v,w_{\alpha})}{|x-y|^{N+s(x,y)p(x,y)}}dx dy
+\int_{B_{\alpha}}\int_{A_{\alpha}}\frac{F_{p(x,y)}(u,w_{\alpha})-F_{p(x,y)}(v,w_{\alpha})}{|x-y|^{N+s(x,y)p(x,y)}}dx dy\\
+&\int_{A_{\alpha}}\int_{B_{\alpha}}\frac{F_{p(x,y)}(u,w_{\alpha})-F_{p(x,y)}(v,w_{\alpha})}{|x-y|^{N+s(x,y)p(x,y)}}dx dy
\end{split}
\end{align*}
\begin{claim}
We claim the following.
 \end{claim}
\begin{equation*}
F_{p(x,y)}(u,w_{\alpha})-F_{p(x,y)}(v,w_{\alpha})\geq 
\begin{cases}
\frac{|w_{\alpha}(x)-w_{\alpha}(y)|^{p^+}}{2^{2}\cdot3^{\frac{p^{+}}{2}}},& \text{if}\;\;{|w_{\alpha}(x)-w_{\alpha}(y)|\leq 1}\\
\frac{|w_{\alpha}(x)-w_{\alpha}(y)|^{p^-}}{2^{2}\cdot3^{\frac{p^{+}}{2}}},&\text{if}\;\;{|w_{\alpha}(x)-w_{\alpha}(y)|> 1}
\end{cases}
\end{equation*}
in $\mathbb{R}^N\times\mathbb{R}^N$.
\begin{proof}
Using $A_{\alpha}=A_{\alpha}^{+}\cup A_{\alpha}^{-}$, $B_{\alpha}=B_{\alpha}^{+}\cup B_{\alpha}^{-}$ and the fact that on $A_{\alpha}^{+}\times A_{\alpha}^{+}$ and $A_{\alpha}^{-}\times A_{\alpha}^{-}$, $w_{\alpha}=(u-v)-\alpha$, we get
\begin{align*}
\begin{split}
F_{p(x,y)}(u,w_{\alpha})-F_{p(x,y)}(v,w_{\alpha})&=|u(x)-u(y)|^{p(x,y)-2}(u(x)-u(y))(w_{\alpha}(x)-w_{\alpha}(y))\\
&-|v(x)-v(y)|^{p(x,y)-2}(v(x)-v(y))(w_{\alpha}(x)-w_{\alpha}(y))\\
&=|u(x)-u(y)|^{p(x,y)-2}(u(x)-u(y))\left[ ((u-v)(x)-\alpha)\right.\\
&\left.-((u-v)(y)-\alpha)\right]-|v(x)-v(y)|^{p(x,y)-2}(v(x)-v(y))\\
&\;\;\;\;[((u-v)(x)-\alpha)-((u-v)(y)-\alpha)]\\
&=|u(x)-u(y)|^{p(x,y)-2}(u(x)-u(y))[((u-v)(x))-((u-v)(y))]\\
&-|v(x)-v(y)|^{p(x,y)-2}(v(x)-v(y))[((u-v)(x))-((u-v)(y)]
\end{split}
\end{align*}
Now on $A_{\alpha}^{+}\times A_{\alpha}^{-}$ and $A_{\alpha}^{-}\times A_{\alpha}^{+},$ we have 
\begin{align*}
\begin{split}
F_{p(x,y)}(u,w_{\alpha})-F_{p(x,y)}(v,w_{\alpha})
&=[|u(x)-u(y)|^{p(x,y)-2}(u(x)-u(y))-|v(x)-v(y)|^{p(x,y)-2}\\
&\;\;\;\;\;(v(x)-v(y))][(u-v)(x)-(u-v)(y)-2\alpha]
\end{split}
\end{align*}
\begin{align*}
\begin{split}
\geq
\begin{cases}
\frac{|(u-v)(x)-(u-v)(y)|^{p^+}-1}{2^{2}\cdot3^{\frac{p^{+}}{2}}}[(u-v)(x)-(u-v)(y)-2\alpha], &\;\;\text{if}\;{|(u-v)(x)-(u-v)(y)|\leq 1}\\
\frac{|(u-v)(x)-(u-v)(y)|^{p^-}-1}{2^{2}\cdot3^{\frac{p^{+}}{2}}}[(u-v)(x)-(u-v)(y)-2\alpha], &\;\;\text{if}\;{|(u-v)(x)-(u-v)(y)| > 1}
\end{cases}
\end{split}
\end{align*}
\begin{align*}
\geq
\begin{cases}
\frac{|(u-v)(x)-(u-v)(y)-2\alpha|^{p^+}-1}{2^{2}\cdot3^{\frac{p^{+}}{2}}}[(u-v)(x)-(u-v)(y)-2\alpha], &\text{if}\;{|(u-v)(x)-(u-v)(y)|\leq 1}\\
\frac{|(u-v)(x)-(u-v)(y)-2\alpha|^{p^-}-1}{2^{2}\cdot3^{\frac{p^{+}}{2}}}[(u-v)(x)-(u-v)(y)-2\alpha], &\text{if}\;{|(u-v)(x)-(u-v)(y)| > 1}
\end{cases}
\end{align*}
\begin{align*}
\MoveEqLeft\hspace{-55mm}=
\begin{cases}
\frac{|(u-v)(x)-(u-v)(y)-2\alpha|^{p^+}}{2^{2}\cdot3^{\frac{p^{+}}{2}}}, &\text{if}\;{|(u-v)(x)-(u-v)(y)|\leq 1}\\
\frac{|(u-v)(x)-(u-v)(y)-2\alpha|^{p^+}}{2^{2}\cdot3^{\frac{p^{-}}{2}}}, &\text{if}\;{|(u-v)(x)-(u-v)(y)| > 1}
\end{cases}
\end{align*}
\begin{align*}
\MoveEqLeft\hspace{-75mm}=
\begin{cases}
\frac{|w_{\alpha}(x)-w_{\alpha}(y)|^{p^+}}{2^{2}\cdot3^{\frac{p^{+}}{2}}}, &\text{if}\;{|w_{\alpha}(x)-w_{\alpha}(y)|\leq 1}\\
\frac{|w_{\alpha}(x)-w_{\alpha}(y)|^{p^-}}{2^{2}\cdot3^{\frac{p^{+}}{2}}}, &\text{if}\;{|w_{\alpha}(x)-w_{\alpha}(y)|>1}
\end{cases}
\end{align*}
Similarly, for $A_{\alpha}^{+}\times B_{\alpha}^{+}, A_{\alpha}^{+}\times B_{\alpha}^{-}, A_{\alpha}^{-}\times B_{\alpha}^{+}$ and $A_{\alpha}^{-}\times B_{\alpha}^{-}$, we get similar inequality. Thus, the claim is proved.
\end{proof}
\noindent Hence, 
\begin{equation*}
A_{p(x,y)}(u,w_{\alpha})-A_{p(x,y)}(v,w_{\alpha})\geq
\begin{cases}
\int_{\mathbb{R}^N}\int_{\mathbb{R}^N}\frac{|w_{\alpha}(x)-w_{\alpha}(y)|^{p^+}}{|x-y|^{N+s(x,y)p(x,y)}}dx dy,&\text{if}\;\;{|w_{\alpha}(x)-w_{\alpha}(y)|\leq 1}\\
\int_{\mathbb{R}^N}\int_{\mathbb{R}^N}\frac{|w_{\alpha}(x)-w_{\alpha}(y)|^{p^-}}{|x-y|^{N+s(x,y)p(x,y)}}dx dy, &\text{if}\;\;{|w_{\alpha}(x)-w_{\alpha}(y)|> 1}
\end{cases}
\end{equation*}
\noindent(i) If $|w_{\alpha}(x)-w_{\alpha}(y)| > 1$ and $\left\|w_{\alpha}\right\|\leq 1$ then
\begin{align*}
\begin{split}
A_{p(x,y)}(u,w_{\alpha})-A_{p(x,y)}(v,w_{\alpha})&\geq \frac{1}{2^{2}\cdot3^{\frac{p^{+}}{2}}}\int_{\mathbb{R}^N}\int_{\mathbb{R}^N}\frac{|w_{\alpha}(x)-w_{\alpha}(y)|^{p^-}}{|x-y|^{N+s(x,y)p(x,y)}}dx dy\\
&\geq \frac{1}{2^{2}\cdot3^{\frac{p^{+}}{2}}}\int_{\mathbb{R}^N}\int_{\mathbb{R}^N}\frac{|w_{\alpha}(x)-w_{\alpha}(y)|^{\alpha^+}}{|x-y|^{N+s(x,y)p(x,y)}}dx dy\\
&\geq \frac{1}{2^{2}\cdot3^{\frac{p^{+}}{2}}}\int_{\mathbb{R}^N}\int_{\mathbb{R}^N}\frac{|w_{\alpha}(x)-w_{\alpha}(y)|^{\alpha(x)}}{|x-y|^{N+s(x,y)p(x,y)}}dx dy\\
&\geq\frac{1}{2^{2}\cdot3^{\frac{p^{+}}{2}}}\left\|w_{\alpha}\right\|^{r^+}
\end{split}
\end{align*}
(ii) If $|w_{\alpha}(x)-w_{\alpha}(y)| > 1$ and $\left\|w_{\alpha}\right\| > 1$ then
\begin{align*}
\begin{split}
A_{p(x,y)}(u,w_{\alpha})-A_{p(x,y)}(v,w_{\alpha})&\geq \frac{1}{2^{2}\cdot3^{\frac{p^{+}}{2}}}\int_{\mathbb{R}^N}\int_{\mathbb{R}^N}\frac{|w_{\alpha}(x)-w_{\alpha}(y)|^{p^-}}{|x-y|^{N+s(x,y)p(x,y)}}dx dy\\
&\geq \frac{1}{2^{2}\cdot3^{\frac{p^{+}}{2}}}\int_{\mathbb{R}^N}\int_{\mathbb{R}^N}\frac{|w_{\alpha}(x)-w_{\alpha}(y)|^{\alpha^+}}{|x-y|^{N+s(x,y)p(x,y)}}dx dy\\
&\geq \frac{1}{2^{2}\cdot3^{\frac{p^{+}}{2}}}\int_{\mathbb{R}^N}\int_{\mathbb{R}^N}\frac{|w_{\alpha}(x)-w_{\alpha}(y)|^{\alpha(x)}}{|x-y|^{N+s(x,y)p(x,y)}}dx dy\\
&\geq\frac{1}{2^{2}\cdot3^{\frac{p^{+}}{2}}}\left\|w_{\alpha}\right\|^{\alpha^-}
\end{split}
\end{align*}
Hence, in all the cases involving when $|w_{\alpha}(x)-w_{\alpha}(y)| \leq 1$ with $\left\|w_{\alpha}\right\| \leq 1$ and $|w_{\alpha}(x)-w_{\alpha}(y)| \leq 1$ with $\left\|w_{\alpha}\right\| > 1$, we obtained $A_{p(x,y)}(u,w_{\alpha})-A_{p(x,y)}(v,w_{\alpha}) > 0$ for $u\neq v$. For $\alpha=0,$ we get
 $A_{p(x,y)}(u,u-v)-A_{p(x,y)}(v,u-v) > 0$ for $u\neq v$. Hence, $A_{p(x,y)}(u,v)$ is strictly monotone operator. Final step is to prove that $A_{p(x,y)}$ is a homeomorphism (refer appendix for proof). We have shown that for every $u\in W_{0}^{s(x,y),p(x,y)}(\overline{\Omega}),$ there exists $\mathbb{A}_{p(x,y)}(u)$ in $W^{-s(x,y),p^{\prime}(x,y)}(\Omega)$ such that $A_{p(x,y)}(u,v)=\langle\mathbb{A}_{p(x.y)}(u),v\rangle$ for all $v\in W_{0}^{s(x,y),p(x,y)}(\overline{\Omega})$. This defines an operator $\mathbb{A}_{p(x,y)}:W_{0}^{s(x,y),p(x,y)}(\overline{\Omega})\rightarrow W^{-s(x,y),p^{\prime}(x,y)}(\Omega)$ which is strictly monotone, continuous, coercive and bounded. Hence, by Browder's theorem \cite{Milota}, we get  $\text{Range}(\mathbb{A}_{p(x,y)})=W^{-s(x,y),p^{\prime}(x,y)}(\Omega)$. Thus, for every $f\in W^{-s(x,y),p^{\prime}(x,y)}(\Omega)$, there exists a unique solution $u\in W_{0}^{s(x,y),p(x,y)}(\overline{\Omega})$ of problem (4.6).
\end{proof}
\noindent We will now find an estimate for the operator $\left(-\Delta_{p(x,y)}^{s(x,y)}\right)^{-1}$ which is a map between\\ $W^{-s(x,y),p^{\prime}(x,y)}(\Omega)$ and $W_{0}^{s(x,y),p(x,y)}(\overline{\Omega})$. Let $u=\left(-\Delta_{p(x,y)}^{s(x,y)}\right)^{-1}(f)$ and $v=\left(-\Delta_{p(x,y)}^{s(x,y)}\right)^{-1}(g)$, where $f,g\in W^{-s(x,y),p^{\prime}(x,y)}(\Omega)$. Then for each $\phi\in W_{0}^{s(x,y),p(x,y)}(\overline{\Omega}),$ we have $$A_{p(x,y)}(u,\phi)-A_{p(x,y)}(v,\phi)=\langle f-g,\phi\rangle.$$ In particular, taking $\phi=(u-v)$ and $\alpha=0$, we get $A_{p(x,y)}(u,u-v)-A_{p(x,y)}(v,u-v)=\langle f-g,\phi\rangle$. For $\left\|w_{\alpha}\right\|\leq 1$ i.e. $\left\|u-v\right\|\leq 1$, we have
$$\frac{1}{2^{2}\cdot3^{\frac{p^{+}}{2}}}\left\|u-v\right\|^{r^{+}}\leq \langle f-g,u-v\rangle\leq \left\|f-g\right\|_{W^{-s(x,y),p^{\prime}(x,y)}(\Omega)}\left\|u-v\right\|$$ which implies that
 $\left\|u-v\right\|^{r^{+}-1}\leq \left(2^{2}\cdot3^{\frac{p^{+}}{2}}\right)\left\|f-g\right\|_{W^{-s(x,y),p^{\prime}(x,y)}(\Omega)}$ and hence
 \begin{equation}
\left\|\left(-\Delta_{p(x,y)}^{s(x,y)}\right)^{-1}(f)-\left(-\Delta_{p(x,y)}^{s(x,y)}\right)^{-1}(g)\right\|\leq \left[\left( 2^{2}\cdot3^{\frac{p^{+}}{2}}\right)\left\|f-g\right\|_{W^{-s(x,y),p^{\prime}(x,y)}(\Omega)}\right] ^{\frac{1}{r^{+}-1}}
\end{equation}
Also, for  $\left\|w_{\alpha}\right\| > 1$ i.e. $\left\|u-v\right\| > 1$, we have\\
$\frac{1}{2^{2}\cdot3^{\frac{p^{+}}{2}}}\left\|u-v\right\|^{\alpha^{-}}\leq \langle f-g,u-v\rangle\leq \left\|f-g\right\|_{W^{-s(x,y),p^{\prime}(x,y)}(\Omega)}\left\|u-v\right\|$ which implies that
$\left\|u-v\right\|^{\alpha^{-}-1}\leq \left(2^{2}\cdot3^{\frac{p^{+}}{2}}\right)\left\|f-g\right\|_{W^{-s(x,y),p^{\prime}(x,y)}(\Omega)}$ and hence
\begin{equation}
\left\|\left(-\Delta_{p(x,y)}^{s(x,y)}\right)^{-1}(f)-\left(-\Delta_{p(x,y)}^{s(x,y)}\right)^{-1}(g)\right\|\leq \left[\left( 2^{2}\cdot3^{\frac{p^{+}}{2}}\right)\left\|f-g\right\|_{W^{-s(x,y),p^{\prime}(x,y)}(\Omega)}\right] ^{\frac{1}{\alpha^{-}-1}}
\end{equation}
Define a map $\Phi_{s(x,y),p(x,y)}:W_{0}^{s(x,y),p(x,y)}(\overline{\Omega})\rightarrow W_{0}^{s(x,y),p(x,y)}(\overline{\Omega})$ by $$\Phi_{s(x,y),p(x,y)}(u)=\left(-\Delta_{p(x,y)}^{s(x,y)}\right)^{-1}\left(\beta|u|^{\alpha(x)-2}u+\lambda f(x,u)\right)\;\forall\;u\in W_{0}^{s(x,y),p(x,y)}(\overline{\Omega}).$$ It can be seen from Lemma 4.5 that this map is well-defined and continuous. 
We aim to show that $\Phi_{s(x,y),p(x,y)}$ is a contraction map. Let $u,v\in W_{0}^{s(x,y),p(x,y)}(\overline{\Omega})$. Define $v_{1}=\beta|u|^{\alpha(x)-2}u+\lambda f(x,u)$ and $v_{2}=\beta|v|^{\alpha(x)-2}v+\lambda f(x,v)$. First we will prove that $v_{1}$ and $v_{2}$ are in $W^{-s(x,y),p^{\prime}(x,y)}(\Omega)$. Using Theorem 2.4, we get for $\phi\in W_{0}^{s(x,y),p(x,y)}(\overline{\Omega})$,
\begin{align*}
\begin{split}
\left\lvert\int_{\Omega}|u|^{\alpha(x)-2}u\phi dx\right\rvert&\leq\int_{\Omega}|u|^{\alpha(x)-1}\phi|dx\leq |u|^{\alpha(x)-1}_{L^{\frac{\alpha(x)}{\alpha(x)-1}}(\Omega)}|\phi|_{L^{\alpha(x)}(\Omega)}\\
&\leq\left(\left\|u\right\|^{\alpha^{-}-1}_{L^{\alpha(x)}(\Omega)}+\left\|u\right\|^{\alpha^{+}-1}_{L^{\alpha(x)}(\Omega)}\right)\left\|\phi\right\|_{W_{0}^{s(x,y),p(x,y)}(\overline{\Omega})}\\
&\leq c_{1}\left\|\phi\right\|_{W_{0}^{s(x,y),p(x,y)}(\overline{\Omega})}
\end{split}
\end{align*}
Hence, $|u|^{\alpha(x)-2}u\in W^{-s(x,y),p^{\prime}(x,y)}(\Omega)$. For the second term of $v_{1}$, using assumption $(f_{4})$ it can be seen that $f(x,u)\in L^{p^'(x,y)}(\Omega)$ and hence $f(x,u)\in W^{-s(x,y),p^{\prime}(x,y)}(\Omega)$. Hence, $v_{1}\in W^{-s(x,y),p^{\prime}(x,y)}(\Omega)$. Repeating the argument for $v_{2}$, we get $v_{2}\in W^{-s(x,y),p^{\prime}(x,y)}(\Omega)$. Let $\left\|u-v\right\| < 3\tau < 1$, where $\tau < 1$ is a sufficiently small positive number. Now using (4.8) we obtained
\begin{equation}
\left\|\Phi_{s(x,y),p(x,y)}(u)-\Phi_{s(x,y),p(x,y)}(v)\right\|\leq \left[\left( 2^{2}\cdot3^{\frac{p^{+}}{2}}\right)\left\|v_{1}-v_{2}\right\|_{W^{-s(x,y),p^{\prime}(x,y)}(\Omega)}\right] ^{\frac{1}{r^{+}-1}}
\end{equation}
Choose $\phi\in W_{0}^{s(x,y),p(x,y)}(\overline{\Omega})$ such that $\left\|\phi\right\|\leq 1$. Using (2.3), (2.4) and $(f_{3})$ we have,
\begin{align}
\begin{split}
|\langle v_{1}-v_{2},\phi\rangle|&\leq \int_{\Omega}\left\lvert\beta|u|^{\alpha(x)-2}u\phi+\lambda f(x,u)\phi-\beta|v|^{\alpha(x)-2}v\phi-\lambda f(x,v)\phi\right\rvert dx\\
&\leq \beta\int_{\Omega}| |u|^{\alpha(x)-2}u-|v|^{\alpha(x)-2}v | |\phi|dx+\lambda\int_{\Omega}|f(x,u)-f(x,v)| |\phi|dx\\
&\leq \beta\int_{\Omega}(\alpha(x)-1)|u-v||\phi|\left(|u|^{\alpha(x)+1}+|v|^{\alpha(x)+1}\right)dx+\lambda\delta\int_{\Omega}|u-v||\phi||u|^{r^{+}-2}dx\\
&+\lambda\delta\int_{\Omega}|u-v||\phi||v|^{r^{+}-2}dx\\
&\leq \beta(\alpha^{+}-1)\int_{\Omega}|u-v||u|^{\alpha(x)-1}|\phi|dx+\beta(\alpha^{+}-1)\int_{\Omega}|u-v||v|^{\alpha(x)-1}|\phi|dx\\
&+2^{r^{+}-2}\lambda\delta\int_{\Omega}|\phi||u-v|\left(|u-v|^{r^{+}-2}+|v|^{r^{+}-2}\right)dx\\
&+2^{r^{+}-2}\lambda\delta\int_{\Omega}|\phi||u-v|\left(|u-v|^{r^{+}-2}+|u|^{r^{+}-2}\right)dx
\end{split}
\end{align}
Applying  generalized H\"{o}lder's inequality (refer \cite{Bahrouni}), embedding result and using the fact that $\|\phi\| \leq 1$ for each term, we get
\begin{align}
\begin{split}
|\langle v_{1}-v_{2},\phi\rangle|&\leq \beta(\alpha^{+}-1)C_{0}\left\|u-v\right\|+\beta(\alpha^{+}-1)C_{1}\left\|u-v\right\|+2^{r^{+}-1}\lambda\delta C_{2}\left\|u-v\right\|^{r^{+}-1}\\
&+2^{r^{+}-2}\lambda\delta C_{3}\left\|u-v\right\|+2^{r^{+}-2}\lambda\delta C_{4}\left\|u-v\right\|\\
&\leq \tilde{C}\left\lbrace 2\beta(\alpha^{+}-1)\left\|u-v\right\|+2^{r^{+}-1}\lambda\delta\left\|u-v\right\|^{r^{+}-1}+2^{r^{+}-1}\lambda\delta\left\|u-v\right\|\right\rbrace 
\end{split}
\end{align}
where $C_{0}, C_{1}, C_{2}, C_{3}, C_{4}$ are constants from embedding and H\"{o}lder's inequality and $\tilde{C}=\\ \text{max}\left\lbrace C_{0}, C_{1}, C_{2}, C_{3}, C_{4}\right\rbrace  > 0$.
For $\phi=\frac{v_{1}-v_{2}}{\left\|v_{1}-v_{2}\right\|}$, (4.12) becomes 
\begin{align}
\begin{split}
\left\|v_{1}-v_{2}\right\|&\leq \tilde{C}\left\|u-v\right\|^{r^{+}-1}\left\lbrace 2\beta(\alpha^{+}-1)\left\|u-v\right\|^{2-r^{+}}+2^{r^{+}-1}\lambda\delta+2^{r^{+}-1}\lambda\delta\left\|u-v\right\|^{2-r^{+}}\right\rbrace \\
&< \tilde{C}\left\|u-v\right\|^{r^{+}-1}\left\lbrace 2\beta(\alpha^{+}-1)(3\tau)^{2-r^{+}}+2^{r^{+}-1}\lambda\delta+2^{r^{+}-1}\lambda\delta(3\tau)^{2-r^{+}}\right\rbrace\\
&= \tilde{C}K(\beta,\lambda,\delta,\tau)\left\|u-v\right\|^{r^{+}-1}
\end{split}
\end{align}
where, $K(\beta,\lambda,\delta,\tau)=\left\lbrace 2\beta(\alpha^{+}-1)(3\tau)^{2-r^{+}}+2^{r^{+}-1}\lambda\delta+2^{r^{+}-1}\lambda\delta(3\tau)^{2-r^{+}}\right\rbrace.$
\noindent Substituting (4.13) in (4.10), we get
\begin{align}
\begin{split}
\left\|\Phi_{s(x,y),p(x,y)}(u)-\Phi_{s(x,y),p(x,y)}(v)\right\|&<  \left[\left( 2^{2}\cdot3^{\frac{p^{+}}{2}}\right)K(\beta,\lambda,\delta,\tau)\tilde{C}\left\|u-v\right\|^{r^{+}-1}\right] ^{\frac{1}{r^{+}-1}}\\
&< \left[\left( 2^{2}\cdot3^{\frac{p^{+}}{2}}\right)K(\beta,\lambda,\delta,\tau)\tilde{C}\right]^{\frac{1}{r^{+}-1}}\left\|u-v\right\|
\end{split}
\end{align}
\noindent Choose $\lambda$ and $\beta$ sufficiently small say $\lambda\in(0,\lambda_{2})$ and $\beta\in (0,\beta_{2})$ such that the coefficient of (4.14) becomes less than 1. Thus $\Phi_{p(x,y),s(x,y)}$ is a contraction map. Therefore there exists a unique fixed point say $u_{2}$. Let $u_{0}=\tau_{2}u$. Define a set $M= \left\lbrace v\in W_{0}^{s(x,y),p(x,y)}(\overline{\Omega}):\left\|v-u_{0}\right\|\leq\tau\right\rbrace$. $M$ being a closed subspace in the Banach space $W_{0}^{s(x,y),p(x,y)}(\overline{\Omega})$ is also complete. By the argument above, since $\Phi_{p(x,y),s(x,y)}$ is a contraction map, there exists a unique fixed point $u_{2}\in M$. Since, $u_{0}$ is in $\mathcal{N}^{-}$ and $I_{\lambda,\beta}(u_{0}) > 0$, so does $u_{2}$. Hence, $u_{2}$ is a nontrivial weak solution of the problem in $\mathcal{N}^{-}$ and Lemma 4.4 is proved. Thus we have proved the existence of two distinct nontrivial solutions of the problem (1.1) in Nehari manifold for $\lambda\in(0,\text{min}(\lambda_{1},\lambda_{2}))$ and $\beta\in(0,\text{min}(\beta_{1},\beta_{2}))$.
\section{Uniform estimate}
In $\mathcal{N}^{-}(\text{or}\;\mathcal{N}^{+})$ we have either $|u(x)| < 1$ or $|u(x)| > 1$.
We will now show that the weak solution belonging to either $\mathcal{N}^{+}$ or $\mathcal{N}^{-}$ is bounded. In both $\mathcal{N}^{+}$ and $\mathcal{N}^{-}$, we have $\left\|u\right\| < 1$. Let $u$ be the weak solution of the problem (1.1). If $|u(x)| < 1$ then clearly $u$ is in $L^{\infty}(\Omega)$. The case when $|u(x)| > 1$ has been proved in appendix.
\section{Appendix}
A map $L:X\rightarrow X^{*}$, where $X$ is a Banach space is said to be of type $(S_{+})$ if $u_{n}\rightharpoonup u$ in $X$ and $\underset{n\rightarrow\infty}{\text{lim sup}}\langle L(u_{n})-L(u),u_{n}-u\rangle \leq 0,$ then $u_{n}\rightarrow u$ in $X$.
\begin{lemma}
The operator $A_{p(x,y)}$ is of type $(S_{+})$ and a homeomorphism.
\end{lemma}
\begin{proof}
Let $u_{n}\in W_{0}^{s(x,y),p(x,y)}(\overline{\Omega})$ be a sequence such that $u_{n}\rightharpoonup u$ in $W_{0}^{s(x,y),p(x,y)}(\overline{\Omega})$ and $\underset{n\rightarrow\infty}{\text{lim sup}}\langle A_{p(x,y)}(u_{n})-A_{p(x,y)}(u),u_{n}-u\rangle \leq 0$. Then by Theorem 2.3, $u_{n}(x)\rightarrow u(x)$ a.e. in $\Omega$. We have also proved that $A_{p(x,y)}$ is a continuous, bounded and strictly monotone operator, so $\underset{n\rightarrow\infty}{\text{lim}}\langle A_{p(x,y)}(u_{n})-A_{p(x,y)}(u),u_{n}-u\rangle=0$. Now, applying Fatou's lemma
\begin{equation}
\underset{n\rightarrow\infty}{\text{lim inf}}\int_{\Omega}\int_{\Omega}\frac{|u_{n}(x)-u_{n}(y)|^{p(x,y)}}{|x-y|^{N+s(x,y)p(x,y)}} dx dy \geq \int_{\Omega}\int_{\Omega}\frac{|u(x)-u(y)|^{p(x,y)}}{|x-y|^{N+s(x,y)p(x,y)}} dx dy.
\end{equation}
From $u_{n}\rightharpoonup u$, we have
\begin{equation}
\underset{n\rightarrow\infty}{\text{lim}}\langle A_{p(x,y)}(u_{n}),u_{n}-u\rangle=\underset{n\rightarrow\infty}{\text{lim}}\langle A_{p(x,y)}(u_{n})-A_{p(x,y)}(u),u_{n}-u\rangle=0
\end{equation}
On computation, we get $\langle A_{p(x,y)}(u_{n}),u_{n}-u\rangle$ is equal to
\begin{equation}
\int_{\Omega}\int_{\Omega}\left(\frac{|u_{n}(x)-u_{n}(y)|^{p(x,y)}}{|x-y|^{N+s(x,y)p(x,y)}}-\frac{|u_{n}(x)-u_{n}(y)|^{p(x,y)-2}(u_{n}(x)-u_{n}(y))(u(x)-u(y))}{|x-y|^{N+s(x,y)p(x,y)}}\right) dx dy
\end{equation}
Now, applying Young's inequality in the second term
\begin{align*}
\begin{split}
&\int_{\Omega}\int_{\Omega}\frac{|u_{n}(x)-u_{n}(y)|^{p(x,y)-2}(u_{n}(x)-u_{n}(y))(u(x)-u(y))}{|x-y|^{N+s(x,y)p(x,y)}} dx dy\\
=&\int_{\Omega}\int_{\Omega}\left(|u_{n}(x)-u_{n}(y)|^{p(x,y)-1}{\left(\frac{1}{|x-y|^{N+s(x,y)p(x,y)}}\right)}^{\frac{1}{p^{\prime}(x,y)}}\right)\cdot\\
&\left(|u(x)-u(y)|{\left(\frac{1}{|x-y|^{N+s(x,y)p(x,y)}}\right)}^{\frac{1}{p(x,y)}}\right)dx dy\\
\leq& \int_{\Omega}\int_{\Omega}\left(\frac{p(x,y)-1}{p(x,y)}\right)\left[\frac{|u_{n}(x)-u_{n}(y)|^{p(x,y)}}{|x-y|^{N+s(x,y)p(x,y)}}\right]dx dy\\
+&\int_{\Omega}\int_{\Omega}\frac{1}{p(x,y)}\left[\frac{|u(x)-u(y)|^{p(x,y)}}{|x-y|^{N+s(x,y)p(x,y)}}\right]dx dy
\end{split}
\end{align*}
Hence, 
\begin{align}
\begin{split}
\langle A_{p(x,y)}(u_{n}),u_{n}-u\rangle &\geq \int_{\Omega}\int_{\Omega}\frac{1}{p(x,y)}\left[\frac{|u_{n}(x)-u_{n}(y)|^{p(x,y)}}{|x-y|^{N+s(x,y)p(x,y)}}-\frac{|u(x)-u(y)|^{p(x,y)}}{|x-y|^{N+s(x,y)p(x,y)}}\right]dx dy\\
&\geq c\int_{\Omega}\int_{\Omega}\left[\frac{|u_{n}(x)-u_{n}(y)|^{p(x,y)}}{|x-y|^{N+s(x,y)p(x,y)}}-\frac{|u(x)-u(y)|^{p(x,y)}}{|x-y|^{N+s(x,y)p(x,y)}}\right]dx dy
\end{split}
\end{align}
where $c > 0$.
Thus, from (6.1), (6.2) and (6.4), we get that
\begin{equation}
\underset{n\rightarrow\infty}{\text{lim}}\int_{\Omega}\int_{\Omega}\frac{|u_{n}(x)-u_{n}(y)|^{p(x,y)}}{|x-y|^{N+s(x,y)p(x,y)}}=\int_{\Omega}\int_{\Omega}\frac{|u(x)-u(y)|^{p(x,y)}}{|x-y|^{N+s(x,y)p(x,y)}}
\end{equation}
From the fact that $u_{n}(x)\rightarrow u(x)$ a.e. in $\Omega$, (6.5) and Brezis-Lieb lemma \cite{Lieb}, we conclude that $u_{n}\rightarrow u$ in $W_{0}^{s(x,y),p(x,y)}(\overline{\Omega})$. Hence, $A_{p(x,y)}$ is a mapping of type $(S_{+})$. We also know that $A_{p(x,y)}$ is surjective (refer \cite{Zed}). So, there exists an inverse map $A_{p(x,y)}^{-1}:W^{-s(x,y),p(x,y)}{\Omega}\rightarrow W_{0}^{s(x,y),p(x,y)}(\overline{\Omega})$. We will now prove that the map $A_{p(x,y)}^{-1}$ is continuous. Let $f_{n},f\in W^{-s(x,y),p(x,y)}{\Omega}$ such that $f_{n}\rightarrow f$ in $W^{-s(x,y),p(x,y)}{\Omega}$. Also let $A_{p(x,y)}^{-1}(f_{n})=u_{n}$ and $A_{p(x,y)}^{-1}(f)=u$ i.e. $A_{p(x,y)}(u_{n})=f_{n}$ and $A_{p(x,y)}(u)=f$. Since, $A_{p(x,y)}$ is coercive so $(u_{n})$ is bounded in $W_{0}^{s(x,y),p(x,y)}(\overline{\Omega})$. By reflexivity of space $W_{0}^{s(x,y),p(x,y)}(\overline{\Omega})$, let $u_{n}\rightharpoonup u_{0}$(say) in $W_{0}^{s(x,y),p(x,y)}(\overline{\Omega})$. It follows that $$\underset{n\rightarrow\infty}{\text{lim}}\langle A_{p(x,y)}(u_{n})-A_{p(x,y)}(u_{0}),u_{n}-u_{0}\rangle=\underset{n\rightarrow\infty}{\text{lim}}\langle f_{n},u_{n}-u\rangle=0.$$ As $A_{p(x,y)}$ is of type $(S_{+})$ so $u_{n}\rightarrow u_{0}$ in $W_{0}^{s(x,y),p(x,y)}(\overline{\Omega})$ and hence $A_{p(x,y)}$ is a homeomorphism.
\end{proof}
\begin{lemma}
Let $u$ be a weak solution to the problem (1.1) with $\|u\| < 1$ and $|u| >1$, then $u\in L^{\infty}(\Omega)$.
\end{lemma}
\noindent We will first state two important results which is useful for proving this lemma. 
\begin{lemma} (refer \cite{Biswas})
Let $u:\Omega\rightarrow\mathbb{R}$ be a function such that $u(x) > 1$ a.e. $x\in \Omega$ and $\eta(.,.)$ be a symmetric real valued function such that $0\leq \eta(x,y) < \infty$ for all $(x,y)\in \Omega\times\Omega$. Suppose, $0\leq\eta_{0}\leq\eta^{-}:=\underset{(x,y)\in\Omega\times\Omega}{\text{inf}}\eta(x,y)$. Then we have the following inequality $|u^{\eta(x,y)}(x)-u^{\eta(x,y)}(y)|\geq|u^{\eta_{0}}(x)-u^{\eta_{0}}(y)|$.
\end{lemma}
\begin{lemma} (refer \cite{Parini})
Let $1 < p < \infty$ and $k\geq 1$. For every $a,b,m\geq 0$, it holds that $|a-b|^{p-2}(a-b)(a_{m}^{k}-b_{m}^{k})\geq \frac{kp^{p}}{(k+p-1)^{p}}|a_{m}^{\frac{k+p-1}{p}}-b_{m}^{\frac{k+p-1}{p}}|^{p}$, where $a_{m}=\text{min}\left\lbrace a,m\right\rbrace$ and $b_{m}=\text{min}\left\lbrace b,m\right\rbrace$.
\end{lemma}
\begin{proof}
We have $|u| > 1$ then $|u|=u^{+}$ or $|u|=u^{-}$. We will first assume that $|u|=u^{+} > 1$. Define a truncation function $u_{m}(x)=\text{min}\left\lbrace m,u^{+}(x)\right\rbrace$. Clearly, $u_{m}(x)\in W_{0}^{s(x,y),p(x,y)}(\overline{\Omega})$. Taking $u_{m}^{k},\; k > 1$ as a test function in (3.1) and using assumption $(f_{4})$ along with Lemma 6.4 , we have
\begin{align}
\begin{split}
&\int_{\mathbb{R}^{N}}\int_{\mathbb{R}^{N}}\frac{kp^{p(x,y)}(x,y)}{(k+p(x,y)-1)^{p(x,y)}}\frac{| u_{m}^{\frac{k+p(x,y)-1}{p(x,y)}}(x)-u_{m}^{\frac{k+p(x,y)-1}{p(x,y)}}(y)|^{p(x,y)}}{|x-y|^{N+s(x,y)p(x,y)}}dx dy\\
&\leq\int_{\mathbb{R}^{N}}\int_{\mathbb{R}^{N}}\frac{|u^{+}(x)-u^{+}(y)|^{p(x,y)-2}(u^{+}(x)-u^{+}(y))(u_{m}^{k}(x)-u_{m}^{k}(y))}{|x-y|^{N+s(x,y)p(x,y)}}dx dy\\
&\leq\beta\int_{\Omega}|u^{+}|^{\alpha(x)}u^{+}u_{m}^{k}(x)+\lambda\int_{\Omega}f(x,u^{+})u_{m}^{k}(x)dx\\
&\leq\beta\int_{\Omega}|u^{+}|^{\alpha(x)+k-1}dx+\lambda\epsilon\int_{\Omega}|u^{+}|^{r^{+}+k-1}dx\\
&< \left(\beta+\lambda\epsilon\right) \int_{\Omega}|u^{+}|^{r^{+}+k-1}dx
\end{split}
\end{align}
Also from Lemma 6.3, we have
$$| u_{m}^{\frac{k+p(x,y)-1}{p(x,y)}}(x)-u_{m}^{\frac{k+p(x,y)-1}{p(x,y)}}(y)|\geq | u_{m}^{\frac{k+r^{+}-1}{r^{+}}}(x)-u_{m}^{\frac{k+r^{+}-1}{r^{+}}}(y)|.$$ From a simple computation it can be seen that
$$\left(\frac{p^{-}}{k+p^{+}-1}\right)^{p^{+}}\int_{\mathbb{R}^{N}}\int_{\mathbb{R}^{N}}\frac{| u_{m}^{\frac{k+r^{+}-1}{r^{+}}}(x)-u_{m}^{\frac{k+r^{+}-1}{r^{+}}}(y)|^{p(x,y)}}{|x-y|^{N+s(x,y)p(x,y)}}dx dy <\left(\beta+\lambda\epsilon\right) \int_{\Omega}|u^{+}|^{r^{+}+k-1}dx.$$
Hence, 
$$\int_{\mathbb{R}^{N}}\int_{\mathbb{R}^{N}}\frac{| u_{m}^{\frac{k+r^{+}-1}{r^{+}}}(x)-u_{m}^{\frac{k+r^{+}-1}{r^{+}}}(y)|^{p(x,y)}}{|x-y|^{N+s(x,y)p(x,y)}}dx dy < \left(\beta+\lambda\epsilon\right)\left(\frac{k+p^{+}-1}{p^{-}}\right)^{p^{+}}\int_{\Omega}|u^{+}|^{r^{+}+k-1}dx.$$
Let $\Lambda=\frac{k+r^{+}-1}{r^{+}}\geq 1$. Using the fact that $u_{m}(x)\rightarrow u^{+}(x)$ a.e. $x\in\Omega$, we obtain
\begin{align*}
\begin{split}
\int_{\mathbb{R}^{N}}\int_{\mathbb{R}^{N}}\frac{| u^{+\Lambda}(x)-u^{+\Lambda}(y)|^{p(x,y)}}{|x-y|^{N+s(x,y)p(x,y)}}dx dy
&\leq \underset{n\rightarrow\infty}{{\underline{\text{lim}}}}\int_{\mathbb{R}^{N}}\int_{\mathbb{R}^{N}}\frac{| u_{m}^{\Lambda}(x)-u_{m}^{\Lambda}(y)|^{p(x,y)}}{|x-y|^{N+s(x,y)p(x,y)}}dx dy\\
&< \left(\beta+\lambda\epsilon\right)\left(\frac{k+p^{+}-1}{p^{-}}\right)^{p^{+}}\int_{\Omega}|u^{+}|^{r^{+}+k-1}dx.
\end{split}
\end{align*}
Now there arises two cases.\\
\textbf{Case I:} $\|u^{+\Lambda}\| < 1$ then 
\begin{align*}
\begin{split}
\|u^{+\Lambda}\|^{p^{+}} &< \left(\beta+\lambda\epsilon\right)\left(\frac{k+p^{+}-1}{p^{-}}\right)^{p^{+}}\int_{\Omega}|u^{+}|^{r^{+}+k-1}dx\\
&\leq\left(\beta+\lambda\epsilon\right)\left(\frac{k+p^{+}-1}{p^{-}}\right)^{p^{+}}\|u^{+}\|^{r^{+}+k-1}_{L^{r^{+}+k-1}(\Omega)}\\
&\leq c_{1}\left(\beta+\lambda\epsilon\right)\left(\frac{k+p^{+}-1}{p^{-}}\right)^{p^{+}}\|u^{+}\|^{r^{+}+k-1}
\end{split}
\end{align*}
where $c_{1}$ is a Sobolev constant. Since, $\|u\| < 1$ so $\|u^{+}\| < 1$ and hence $$\|u^{+\Lambda}\|^{p^{+}}\leq c_{1}\left(\beta+\lambda\epsilon\right)\left(\frac{k+p^{+}-1}{p^{-}}\right)^{p^{+}}\|u^{+}\|^{r^{+}-1}.$$  Also by the embedding of $W_{0}^{s(x,y)p(x,y)}(\overline{\Omega})$ into $L^{\Lambda p^{+}}(\Omega)$, we have
\begin{align*}
\begin{split}
\|u^{+}\|_{\Lambda^{2}p^{+}}^{\Lambda p^{+}}&\leq C^{\prime}\left(\beta+\lambda\epsilon\right)\left(\frac{k+p^{+}-1}{p^{-}}\right)^{p^{+}}\|u^{+}\|^{r^{+}-1}\\
&\leq C^{\prime}\left(\beta+\lambda\epsilon\right)\left(\frac{k+p^{+}-1}{p^{-}}\right)^{p^{+}}
\end{split}
\end{align*}
which further implies that
$\|u^{+}\|_{\Lambda^{2}p^{+}}\leq \left(C^{\prime}\left(\beta+\lambda\epsilon\right)\right)^{\frac{1}{\Lambda p^{+}}}\left(\frac{k+p^{+}-1}{p^{-}}\right)^{\frac{1}{\Lambda}}$. Now letting $\Lambda\rightarrow\infty,$ we get $u^{+}\in L^{\infty}(\Omega)$.\\
\textbf{Case II:} If $\|u^{+\Lambda}\| > 1$, then computing in a similar way we get
\begin{align*}
\begin{split}
\|u^{+}\|_{\Lambda^{2}p^{-}}^{\Lambda p^{-}}&\leq C^{\prime}\left(\beta+\lambda\epsilon\right)\left(\frac{k+p^{+}-1}{p^{-}}\right)^{p^{+}}\|u^{+}\|^{r^{+}-1}\\
&\leq C^{\prime}\left(\beta+\lambda\epsilon\right)\left(\frac{k+p^{+}-1}{p^{-}}\right)^{p^{+}}
\end{split}
\end{align*}
which implies that $\|u^{+}\|_{\Lambda^{2}p^{-}}\leq \left(C^{\prime}\left(\beta+\lambda\epsilon\right)\right)^{\frac{1}{\Lambda p^{-}}}\left(\frac{k+p^{+}-1}{p^{-}}\right)^{\frac{p^{+}}{\Lambda p^{-}}}$. Again letting $\Lambda\rightarrow\infty,$ we get $u^{+}\in L^{\infty}(\Omega)$. Repeating the same process for $u^{-}$, we get $u^{-}\in L^{\infty}(\Omega)$. Since, $u=u^{+}-u^{-}$ so $u\in L^{\infty}(\Omega)$. Thus, we get the desired result.
\end{proof}
\section{Conclusion}
We have shown the existence of two nontrivial weak solutions for the problem (1.1) using the variational method and the Banach fixed point theorem in the Nehari manifold for a particular range of $\beta$ and $\lambda$. We have also proved that both the solutions are in $L^{\infty}(\Omega)$.
\section*{Acknowledgement}
The author Amita Soni thanks the Department of
Science and Technology (D. S. T), Govt. of India for financial
support. Both the authors also acknowledge the facilities received
from the Department of mathematics, National Institute of Technology Rourkela. On behalf of all authors, the corresponding author states that there is no conflict of interest.

 {\sc Amita Soni} and {\sc D. Choudhuri}\\
Department of Mathematics,\\
National Institute of Technology Rourkela, Rourkela - 769008,
India\\
e-mails: soniamita72@gmail.com and dc.iit12@gmail.com.

\end{document}